\title{Cyclotomic polynomials with prescribed height and prime number theory}
\author{Alexandre Kosyak, Pieter Moree, Efthymios Sofos and Bin Zhang}
\documentclass[12pt]{article}
\usepackage{amssymb, latexsym, amsfonts, url}
\usepackage{amsmath}
\usepackage{amsthm}
\usepackage{dsfont}
\def\@ptsize{2}
\usepackage{tikz}

\newtheorem{Thm}{Theorem}
\newtheorem{Con}{Conjecture}
\newtheorem{Lem}{Lemma}
\newtheorem{Quest}{Question}

\newtheorem{Def}{Definition}
\newtheorem{cor}{Corollary}
\newtheorem{rem}{Remark}
\newtheorem{ex}{Example}

\newtheorem{Prop}{Proposition}
\begin{document}
\date{}
\maketitle
{\def\thefootnote{}
}
\begin{abstract}
\noindent
Given any positive integer $n,$ let $A(n)$ denote the height of 
the $n^{\text{th}}$ cyclotomic polynomial, that is its
maximum coefficient in absolute value.
It is well known that $A(n)$ is unbounded. We conjecture that every natural
number can arise as value of $A(n)$ 
and prove this assuming that
for every pair of consecutive primes $p$ and $p'$ with $p\ge 127$
we have $p'-p<\sqrt{p}+1.$
We also conjecture that
every natural number occurs as the maximum coefficient of some cyclotomic polynomial and show that this
is true if Andrica's
conjecture holds, i.e., that 
$\sqrt{p'}-\sqrt{p}<1$ always holds. This is the first time, as far as the authors 
know, that a connection between prime gaps
and cyclotomic polynomials is uncovered.
Using a result of 
Heath-Brown on prime gaps we show unconditionally that every natural number $m\le x$ occurs
as $A(n)$ value with at most $O_{\epsilon}(x^{3/5+\epsilon})$ exceptions.
On the Lindel\"of Hypothesis we show there are at most 
$O_{\epsilon}(x^{1/2+\epsilon})$ exceptions and study them further 
by  using deep work of Bombieri--Friedlander--Iwaniec 
  on the distribution of primes in arithmetic progressions
  beyond the square-root barrier.
\end{abstract}
\section{Introduction}
Let $n\ge 1$ be an integer. The $n^{th}$ cyclotomic polynomial 
\begin{equation*}
\Phi_n(x)=\sum_{j=0}^{\varphi(n)}a_n(j)x^j,
\end{equation*}
is a polynomial of degree $\varphi(n),$ with $\varphi$ Euler's totient function.
For $j>\varphi(n)$ we put $a_n(j)=0.$
The coefficients $a_n(j)$ are usually very small. In the $19^{th}$ century
mathematicians even
thought that they are always $0$ or $\pm1$. The first counterexample to this claim occurs at $n = 105$: indeed, $a_{105}(7) = -2$. The number $105$ is the smallest ternary number (see Definition 
\ref{def:ternary}) and these will play
a major role in this article.
Issai Schur proved that every negative even number occurs
as a cyclotomic coefficient.
Emma Lehmer \cite{EmmaLehmer} reproduced his unpublished proof.
Schur's argument is easily adapted to show that 
\emph{every} integer occurs as a cyclotomic coefficient; see Suzuki \cite{Suzi} or Moree and
Hommersom \cite[Proposition 5]{MorHo}.
Let $m\ge 1$ be given. 
Ji, Li and Moree \cite{JLM} adapted Schur's argument and proved that
\begin{equation}
    \label{jilimo}
\{a_{mn}(j): n\ge 1,j\ge 0\}=\mathbb Z.
\end{equation}
Fintzen \cite{Fintzen} determined the set of all cyclotomic coefficients
$a_n(j)$ with $j$ and $n$ in prescribed arithmetic progression, thus 
generalizing \eqref{jilimo}.

We put  
$$A(n)=\max_{k\ge 0}|a_{n}(k)|,\,\,\,{\cal A}=\cup_{n\in \mathbb N}A(n),\,\,\,A\{n\}=\{a_{n}(k):k\ge 0\},$$
in particular $A(n)$ is the  
height of the cyclotomic polynomial
$\Phi_n.$ 

It is a classical result that if $n$ has at most two distinct odd
prime factors, then $A(n)=1,$ cf.\ Lam and 
Leung \cite{LL}. The first non-trivial case arises where $n$ has
precisely three distinct odd prime divisors and thus is of the form $n=p^eq^fr^g,$
with $2<p<q<r$ prime numbers. It is easy to
deduce that 
$A\{p^eq^fr^g\}=A\{pqr\}$
using elementary properties of cyclotomic polynomials (as given
for example in \cite[Lemma 2]{MorHo}). It thus suffices to consider only the 
case where $e=f=g=1$ and so $n=pqr.$  This motivates the following definition.
\begin{Def}
\label{def:ternary}
A cyclotomic polynomial $\Phi_n(x)$ is said to be 
\emph{ternary} if $n = pqr$, with $2<p<q<r$ primes. In this case we call the integer $n=pqr$ \emph{ternary}.  
We put ${\cal A}_t=\{A(n):n\text{~is~ternary}\}.$
\end{Def}
Note that ${\cal A}_t\subseteq {\cal A}.$ In this article we address the nature of the 
sets ${\cal A}, {\cal A}_t$ and  $\mathcal A_{opt}$ (see Definition \ref{def:optimal} below). 
\begin{Con}
\label{nul}
We have ${\cal A}={\mathbb N},$ that is for any given natural number $m$ there is a 
cyclotomic polynomial having height $m.$
\end{Con}

\begin{Con}
\label{een}
We have ${\cal A}_t={\mathbb N},$ that is for any given natural number $m$ there is a ternary
$n$ such that $\Phi_n$ has height $m.$
\end{Con}

The argument of Schur cannot be adapted to resolve Conjecture \ref{nul}, as it allows one to 
control only the coefficients in a tail of a polynomial that quickly becomes very
large if we want to show that some larger number occurs as a coefficient, and typically will 
have much larger coefficients than the coefficient constructed.
Instead, we will make use of various properties of ternary cyclotomic polynomials. This
class of cyclotomic polynomials has been intensively studied as it is the simplest 
one where
the coefficients display non-trivial behavior.
For these we still have
$\{a_{n}(j): n{\rm~is~ternary},\,j\ge 0\}=\mathbb Z,$
as a consequence of the following result.
\begin{Thm}[Bachman, \cite{B2}]
\label{t.Bach04}
For every odd prime $p$ there exists an infinite family of polynomials $\Phi_{pqr}$ such that 
$A\{pqr\}=[-(p\!-\!1)/2,(p\!+\!1)/2]\cap{\mathbb Z}$ and another one
such that 
$A\{pqr\}=[-(p\!+\!1)/2,(p\!-\!1)/2]\cap{\mathbb Z}.$
\end{Thm}
\indent If $n$ is ternary, then $A\{n\}$ consists of consecutive integers. Moreover, we 
have $|a_n(j+1)-a_n(j)|\le 1$ for $j\ge 0$; see Gallot and Moree \cite{buur}.
Note that for each of the members $\Phi_{pqr}$ of the two families the cardinality of $A\{pqr\}$ is $p+1.$ 
This is not always the case for arbitrary ternary $n$ and 
even best possible in the sense that $\# A\{pqr\}\le p+1$ for arbitrary ternary $n$
(by \cite[Corollary 3]{B1}).
\begin{Def}
\label{def:optimal}
If the cardinality of $A\{pqr\}$ is exactly $p+1,$ we say that $\Phi_{pqr}$ is ternary optimal and
call $n=pqr$ optimal. We denote 
the set of all $A(n)$ with $n$ optimal by $\mathcal A_{opt}$.
\end{Def}

Note that the bound for the size of ${\cal A}\{pqr\}$ depends only on the smallest prime
factor, $p$. Similarly, it has been known since the 19th century that $A(pqr)\leq p-1$.

We expect the following to be true
regarding ternary optimal polynomials.
\begin{Con}
\label{opt}
We have ${\cal A}_{opt}={\mathbb N}\backslash \{1,5\}.$
\end{Con}
We will see that this conjecture is closely related to the following
prime number conjecture we propose (with $p_n$ the 
$n^{\text{th}}$
prime number).
\begin{Con}
\label{Pi.Andr}
Let $n\ge 31$ (and so $p_n\ge 127$). Then
\begin{equation}
\label{eq:primeinequality}
p_{n+1}-p_n< \sqrt{p_n}+1. 
\end{equation}
\end{Con}
Although prime gaps 
$d_n:=p_{n+1}-p_n$
have been studied {\it in extenso} in the literature,  
we have not come across this particular conjecture. 
It is a bit stronger than Andrica's conjecture (see Visser \cite{Visser} for some numerics). 
\begin{Con}
\label{Andr}{\rm (Andrica's conjecture).}
For $n\ge 1$,  
$p_{n+1}-p_n< \sqrt{p_n}+\sqrt{p_{n+1}}$, 
or equivalently $\sqrt{p_{n+1}}-\sqrt{p_n}<1$, or equivalently
$p_{n+1}-p_n< 2\sqrt{p_n}+1$.
\end{Con}

Both conjectures seem to be far out of reach, as under RH the best
result is due to Cram\'er \cite{Cramer} who showed in 
1920 that $d_n=O(\sqrt{p_n}\log p_n)$. More explicitly, 
Carneiro et al.\,\cite{CMS} showed under RH that $d_n\le \frac{22}{25}\sqrt{p_n}\log p_n$
for every $p_n>3$.

There is a whole range of conjectures on gaps between consecutive primes. The most famous one is 
Legendre's that there is a prime between consecutive squares is a bit weaker, but for example 
Firoozbakht's conjecture that
$p_{n}^{1/n}$ is a strictly decreasing function of $n$ 
is much stronger.
Firoozbakht's conjecture implies
that $d_{n}<(\log p_{n})^{2}-\log p_{n}+1$ for all $n$ 
sufficiently large (see Sun \cite{Sun}), contradicting a heuristic 
model; see 
Banks et al.\,\cite{BFT}, suggesting that given any $\epsilon>0$ there are infinitely many 
$n$ such that $d_n>(
2e^{-\gamma}
-\epsilon)(\log p_n)^2,$
with $\gamma $ Euler's constant.
This is in line with Cram\'er's \cite{Cramer36} conjecture of 1936 that 
$$0<\liminf_{x \to \infty}
\frac{\max\{d_n: p_n \leq x \} }{ (\log x)^2}
\leq 
\limsup_{x \to  \infty}
\frac{\max\{d_n: p_n \leq x \} }{ (\log x)^2}
<\infty,
$$ who gave heuristical arguments in support of
this assertion.
This is in line with the famous conjecture that 
$$0<\liminf_{x \to \infty}
\frac{\max\{d_n: p_n \leq x \} }{ (\log x)^2}
\leq 
\limsup_{x \to  \infty}
\frac{\max\{d_n: p_n \leq x \} }{ (\log x)^2}
<\infty,
$$ stated in 1936 by Cram\'er \cite{Cramer36}, who 
also provided heuristic arguments in support of it.
His conjecture implies that
$d_n = O\left((\log p_n)^2\right)$, which if 
true, clearly
shows that the claimed bound in Conjecture~\ref{Pi.Andr}
holds for all sufficiently large $n$.
Further work on $d_n$ can be found in \cite{BFT, longgaps, MR1349149}.

We denote the set of natural numbers $\le h$ 
by $\mathbb N_{h}.$  

\begin{Thm}
\label{thm:main1}
Let $h$ be 
an integer such that \eqref{eq:primeinequality} holds for $127\le p_n<2h$. Then
$$
\mathbb N_h\subseteq \mathcal A_t \subseteq 
\mathcal A,\,\,\,\mathbb N_h\backslash \{1,5\}\subseteq \mathcal A_{opt}.$$
Moreover, $1,5\not \in {\mathcal A}_{opt}$.
\end{Thm}
\begin{cor}
If Conjecture $4$ is true, then so are
Conjectures $1,2$ and $3$.
\end{cor}

Theorem \ref{thm:main1} is in essence a consequence of a result of Moree and 
Ro\c su \cite{Eugenia} (Theorem \ref{t.Eugenia} below) generalizing
Theorem \ref{t.Bach04}, as we shall see 
in \S\,\ref{4conjectures}. 

A lot of numerical work on large gaps has
been done (see the website \cite{Nicely}). 
This can be used to infer that the
inequality 
\eqref{eq:primeinequality} holds
whenever $127\le p_n\le 2\cdot 2^{63}\approx 1.8\cdot 10^{19}$; see Visser \cite{Visser}.
This in combination with Theorem 
\ref{thm:main1} leads to the following proposition.
\begin{Prop}
Every integer up to $9\cdot 10^{18}$ occurs as the height
of some ternary cyclotomic polynomial.
\end{Prop}
The following theorem is the main result of our paper. Its proof rests on combining  
Lemma \ref{l.pi1}b, the key lemma used to prove
 Theorem  \ref{thm:main1}, with 
deep work 
 by Heath-Brown \cite{primegap} and Yu \cite{MR1374401} on gaps between primes.
\begin{Thm}
\label
{thm:main2}
Almost all positive integers occur
as the height of an optimal ternary cyclotomic polynomial. Specifically, for any fixed $\epsilon>0$,
the number of positive 
integers $\le x$ that do not occur
as a height of an optimal ternary cyclotomic polynomial is $\ll_{\epsilon} x^{3/5+\epsilon}.$ Under the Lindel\"of Hypothesis this number
is $\ll_{\epsilon} x^{1/2+\epsilon}.$
\end{Thm}
\noindent (Readers unfamiliar with the Lindel\"of Hypothesis 
are referred to the paragraph $\S\, 3$
before the statement of 
Lemma \ref{thm:yu2}.)
\par In addition to 
Conjecture 
\ref{Pi.Andr}, there are two further prime number conjectures of relevance for the topic at hand:
Conjecture 5, that we have
not come across in the literature, and Andrica's conjecture (Conjecture \ref{Andr}).
\begin{Con}
\label
{Pi.2tuplet.1}
Let $h>1$ be odd. There exists a prime
$p\ge 2h-1,$ such that $1+(h-1)p$
is a prime too.
\end{Con}
The widely believed Bateman--Horn conjecture 
\cite{Al-ZomFukGarcia} implies that given an odd $h>1,$ there
are infinitely many primes $p$ such that $1+(h-1)p$ is a prime too, and thus 
Conjecture \ref{Pi.2tuplet.1} is a weaker version of this.

\begin{Thm}
\label
{thm:main3}
If Conjecture \ref{Pi.2tuplet.1} holds true, then
${\cal A}_t$ contains all odd natural numbers. Unconditionally ${\cal A}_t$ contains a positive fraction of 
all odd natural numbers.
\end{Thm}
The first assertion is a consequence of work of 
Gallot, Moree and Wilms \cite{GMW} and involves ternary
cyclotomic polynomials that are not optimal. The second makes use
of deep work of Bombieri, Friedlander and Iwaniec \cite{MR891581} 
on the level of distribution 
of primes in  arithmetic progressions with
 fixed residue and varying moduli.
The level of distribution that is needed here goes beyond the square root barrier (that is studied in the Bombieri--Vinogradov theorem, for example)
and this is due to the condition $p\ge 2h-1$ in Conjecture~\ref{Pi.2tuplet.1}; see Remark~\ref{rem:1/2bombfriediwa} for more details.
As far as we know, this is the first time 
that this kind of level of distribution is used in the subject of cyclotomic coefficients (see \S\,\ref{s:momentstwo} for the details). We would like to point out though that
Fouvry~\cite{Fouvry}
has used the classical Bombieri--Vinogradov theorem in a rather different way and context, namely, for studying the number of 
nonzero coefficients of cyclotomic polynomials $\Phi_n$ with $n$ having two distinct prime factors.

In the final section we consider cyclotomic polynomials with prescribed maximum or minimum
coefficient. 
We will prove the following result.
\begin{Thm}
\label{thm:Andrica}
Andrica's conjecture 
implies that every natural number occurs
as the maximum coefficient of some cyclotomic polynomial.
\end{Thm}

That prime numbers play such an important role in our approach is a consequence of working with ternary
cyclotomic polynomials. One would want to work with $\Phi_n$ with $n$ having at least four prime
factors, however this leads to a loss of control over the behaviour of the coefficients in general
and the maximum in particular.

\section{More on ternary cyclotomic polynomials}
\label{4conjectures}
Given any $m\ge 1,$ Moree and Ro\c su \cite{Eugenia} 
constructed infinite families of ternary optimal $\Phi_{pqr}$ such
that $A(pqr)=(p+1)/2+m,$ provided that $p$ 
is large enough in terms of $m.$ This result, Theorem \ref{t.Eugenia} below,
allows one to show that for $p\ge 11$ there are cyclotomic polynomials having heights
$(p+1)/2+1,\ldots,(p+1)/2+k$, with $k$ an integer close to $\sqrt{p}/2$. If the
gaps between consecutive primes are always small enough, these heights cover all integers 
large enough and 
this would allow one to
prove Conjecture \ref{een}. If large prime gaps do occur, then we 
are led to study the total length of prime gaps large enough up to
$x$ (cf.\ $E(x)$ in Lemma \ref{lem:Ex}). Conveniently 
for us a good upper bound for this was recently obtained
by Heath-Brown \cite{primegap}; see Lemma \ref{t.HB}.\\
\indent The remainder of this section is devoted to deriving consequences of Theorem 
\ref{t.Eugenia} and proving Theorem \ref{thm:main1}.
\begin{Thm} {\rm (Moree and Ro\c su \cite[Theorem 1.1]{Eugenia})}.
\label{t.Eugenia}
Let 
$p \geq 4m^2+2m+3$ be a prime, with $m\ge 1$ any 
integer. Then there exists an infinite sequence of prime pairs $\{(q_j , r_j )\}_{j=1}^\infty$ 
with $q_j<q_{j+1}, pq_j < r_j$,
such that
$$
A\{pq_j r_j\}=\left\{-\frac{(p-1)}{2}+ m,\ldots,\frac{p + 1}{2} + m\right\}.
$$
\end{Thm}
\noindent We note that the two families of 
Theorem \ref{t.Bach04} are also infinite in the sense of this theorem. Thus Theorem \ref{t.Eugenia} also holds for $m=0$.\\
\indent Put
\begin{equation}
\label{R}
\mathcal R=\Big\{\frac{p+1}{2}+m: p{\rm ~is~a~prime}, \,m\ge 0,\, 4m^2+2m+3\le p\Big\}.
\end{equation}

\begin{Lem}
\label{l:Aopt}
We have $\mathcal R\subseteq \mathcal A_{opt}.$
\end{Lem}
\begin{proof}
For the elements of $\mathcal R$ with $m=0$ this follows from Theorem \ref{t.Bach04}, for 
those with $m\ge 1$ it
follows from Theorem \ref{t.Eugenia}.
\end{proof}

\begin{Lem}
\label{l:upto2h}
If $p_{n+1}-p_n< \sqrt{p_n}+1$ holds for $127\le p_n<2h$ with $h$ an
integer, then we have
$\mathbb N_h\backslash \{1,5,63\}\subseteq \mathcal R.$
\end{Lem}
The proof is a consequence of part a)
of the following lemma 
and the computational observation that $1,5$ and $63$ are
the only natural numbers $<64$ that are not in $\mathcal R.$

By $\lfloor r \rfloor$ we denote the entire part of 
 a real number $r$.
\begin{Lem} 
\label{l.pi1} Let $n\ge 5.$ \\
{\rm a)} If $p_{n+1}-p_n<\sqrt{p_n}+1,$ then
 $I_n\cap \mathbb N\subset \mathcal R$, where $I_n:=\textstyle[\frac{p_n+1}{2},
\frac{p_{n+1}-1}{2}]$.\\
{\rm b)} If $p_{n+1}-p_n\ge \sqrt{p_n}+1$, then there are at most
\begin{equation}
\label{eq:floor}
 \lfloor (p_{n+1}-p_n-\sqrt{p_n}+1)/2\rfloor
\end{equation} 
 integers in the interval $I_n$ that are not in $\mathcal R.$ 
\end{Lem}
\begin{proof}
The assumption on $n$ implies that $p_n\ge 11.$ Put $z_n=(\sqrt{p_{n}}-1)/2.$
Note that $4z_n^2+2z_n+3=p_n-\sqrt{p_n}+3<p_n.$ As 
$4x^2+2x+3$ is increasing for $x\ge 0,$
the inequality $4x^2+2x+3<p_n$ is satisfied for every real 
number $0\le x\le z_n.$ In particular it is satisfied for $x=m_n,$ with 
$m_n$ the unique integer in the interval $[z_n-1,z_n].$
Thus $m_n\ge (\sqrt{p_{n}}-3)/2$ and 
$4m_n^2+2m_n+3\leq p_n$. 
It follows that
$$
\Big[\frac{p_n+1}{2},\frac{p_n+1}{2}+m_n\Big]\cap \mathbb N\subseteq \mathcal R.
$$
As $(p_{n+1}+1)/2$ is clearly in $\mathcal R$, part a) 
follows if we can show that the final number $(p_n+1)/2+m_n$ is at least
$(p_{n+1}-1)/2.$
Since both numbers are integers we can express this as 
 $(p_n+1)/2+m_n>(p_{n+1}-3)/2.$
The validity of this inequality is obvious, since
$$\frac{p_n+1}{2}+m_n\ge \frac{p_n+1}{2}+\frac{\sqrt{p_n}-3}{2}>
\frac{p_{n+1}-3}{2},$$
where the second inequality is a consequence
of our assumption 
$d_n< \sqrt{p_n}+1$.\\
\indent Part b)
follows on noting that the number of 
integers of $\mathcal R$ that are not 
in $I_n$ is bounded above by $d_n/2-1-m_n,$ 
which we see is bounded above by the integer in \eqref{eq:floor} on using $m_n\ge (\sqrt{p_{n}}-3)/2$.
\end{proof}

Since we believe that \eqref{eq:primeinequality} holds for all $p_n\ge 127,$ Lemma \ref{l:upto2h}
leads us to make the following conjecture.
\begin{Con}
\label{c.1.5.63}
We have $\mathcal R=\mathbb N\backslash \{1,5,63\}.$
\end{Con}
The numbers $1,5$ and $63$ are special in our story.
\begin{Lem}
\label{l:1563}
The integers $1$ and $5$ are in $\mathcal A_t\subseteq \mathcal A,$ but not in $\mathcal A_{opt}.$
The integer $63$ is in $\mathcal A_{opt}\subset \mathcal A_t\subseteq \mathcal A,$ but not in
$\mathcal R.$
\end{Lem}
\begin{proof}
If $pqr$ is optimal, then $A(pqr)\ge (p+1)/2\ge 2$ and so $1\not\in  \mathcal A_{opt}$.
It is also easy to see that there is no optimal $pqr$ such that $A(pqr)=5$. If such an optimal $pqr$ would exist, then
as $A(pqr)\le 3$ for $p\le 5$ and $A(pqr)\ge 6$ for $p\ge 11$ (for an optimal 
$pqr$), this would force $p=7$ and $A\{7qr\}=[-5,2]\cap \mathbb Z$ or
$A\{7qr\}=[-2,5]\cap \mathbb Z,$ contradicting the result of
Zhao and Zhang \cite{ZZ1} that 
$A\{7qr\}\subseteq [-4,4]\cap \mathbb Z$.

The number $63$ is in  $\mathcal A_{opt}$.
This follows on applying Theorem 3.1 of \cite{Eugenia}. 
The obvious approach is to consider the largest prime
$p$ such that $(p+1)/2<63,$ which is $p=113,$ and take 
$l=11$ (here and below we use the notation of
Theorem 3.1).
For this combination the result does not
apply, unfortunately. However, it does for $p=109$ and 
$l=15$, in which case we obtain 
$A\{109\cdot 6803\cdot 12084113\}=[-46,\ldots,63]\cap \mathbb Z$ 
(with $q=6803$, $\rho=2870,$ $\sigma=62,$ $s=46$, $\tau=18,$ $w=45$, $r_1=12084113$). 
\end{proof}

\begin{proof}[ Proof of Theorem \ref{thm:main1}]
This follows on 
combining Lemmas \ref{l:Aopt}, \ref{l:upto2h} and \ref{l:1563}. 
\end{proof}

\section{Proof of Theorem \ref{thm:main2}}
\label{sec:gaps}
The goal of this section is to prove Theorem \ref{thm:main2}. The quantity
of central interest, $N(x),$ is 
defined below.
\begin{Def}
The number of
integers $\le x$ that does not occur
as a height of an optimal ternary cyclotomic polynomial is denoted by $N(x)$.
\end{Def}

\begin{Lem}
\label{lem:Ex}
We have $N(x)\le E(2x)/2+O(1),$ where
$$
E(x)=\sum_{p_n\le x
\atop
p_{n+1}-p_n\ge \sqrt{p_n}+1}(p_{n+1}-p_n-\sqrt{p_n}+1).
$$
\end{Lem}
\begin{proof}
By Lemma \ref{l:Aopt} it suffices to bound above the number of integers $\le x$ that are not in
$\mathcal R.$ By Lemma \ref{l.pi1}b) this cardinality, in turn, is bounded above by $E(2x)/2+O(1).$ 
\end{proof}
If Cram\'er's conjecture $d_n = O\left((\log p_n)^2\right)$ holds true, then this lemma implies that $N(x)=O(1)$.

Heath-Brown \cite{primegap} recently proved the following result which gives an upper bound for  $E(x) $.
\begin{Lem}[Heath-Brown]
\label{t.HB}
We have 
$$\sum_{p_n\le x\atop p_{n+1}-p_n\ge \sqrt{p_n}}(p_{n+1}-p_n)\ll_{\epsilon} x^{3/5+\epsilon}.$$
\end{Lem}
\begin{proof}[Proof of the conditional bound of Theorem \ref{thm:main2}] This follows on 
combining the latter upper bound for 
$E(x)$ with Lemma \ref{lem:Ex}.
\end{proof}

In order to complete the proof of Theorem~\ref{thm:main2} we need to
improve the exponent $3/5$ in Lemma~\ref{t.HB}
to $1/2$, conditionally on 
the Lindel\"of Hypothesis.
The Lindel\"of Hypothesis states that 
for all fixed $\epsilon>0$ 
we have  $$ \zeta(1/2 +it)=O_\epsilon
\left( t^\epsilon \right ) , \,t \in \mathbb R, \,t>1,$$
where as usual $\zeta $ denotes the
Riemann zeta function.
It is well-known that the Riemann Hypothesis implies the Lindel\"of Hypothesis, but not vice versa.
There is a large body of work concerning the Lindel\"of Hypothesis (see, for example, the recent work of Bourgain~\cite{MR3556291}), however, it is still open. 
  
We will make use of the following  result of Yu \cite{MR1374401}.
\begin{Lem}
[Yu]
\label{thm:yu2} 
Fix any $\epsilon>0$.
Under the Lindel\"of Hypothesis we have
$$ \sum_{p_n \leq x } (p_{n+1} - p_n)^2 
\ll_{\epsilon} x^{1 +\epsilon}.$$
\end{Lem}
From it one can easily derive a conditional
improvement of Lemma \ref{t.HB}.
\begin{Lem}
\label{lem:rtyuhb23}
Assume the Lindel\"of Hypothesis and fix any 
$\epsilon>0$. Then  we have
\[
\sum_{p_n\leq x \atop p_{n+1} -p_n \geq \sqrt{p_n} }
(p_{n+1} -p_n )
\ll_{\epsilon}
x^{1/2+\epsilon}
.\]
\end{Lem} 
\begin{proof}
Using dyadic division of the interval $[1,x]$ we obtain 
 $$
 \sum_{p_n \leq x  \atop 
 d_n\ge \sqrt{p_n}  } 
 d_n
 \ll
( \log x) \, \max_{1\leq y\leq x }
 \sum_{y<p_n \leq 2y   \atop 
 d_n\ge \sqrt{p_n}  } 
 d_n
 \leq
 ( \log x) 
 \max_{1\leq y\leq x }
 \sum_{y<p_n \leq 2y   \atop 
 d_n\ge \sqrt{p_n}  } 
 \frac{
  d_n^2
 }{\sqrt{p_n}},$$ 
 which by Lemma~\ref{thm:yu2}
 is at most 
 $$
 \hskip 3,5 cm
 (\log x ) \,
 \max_{1\leq y\leq x }
\frac{1}{\sqrt{y}}
 \sum_{y<p_n \leq 2y   } 
  d_n^2
  \ll_{\epsilon} x^{1/2+\epsilon}. 
   \hskip 3,5 cm\qedhere
   $$
  \end{proof}

 \begin{proof}[Proof of the conditional bound of Theorem \ref{thm:main2}]
This follows on 
combining Lemma \ref{lem:Ex} with Lemma 
\ref{lem:rtyuhb23}.
\end{proof}
\begin{rem}
{\rm Although it is not required for the applications in the present paper, one can prove  slightly
stronger variants of Lemmas \ref{t.HB} and \ref{lem:rtyuhb23}, namely, 
that for every fixed $C>0$ and $\epsilon>0,$ we have 
\[
\sum_{p_n\leq x \atop p_{n+1} -p_n \geq C \sqrt{p_n} }
(p_{n+1} -p_n )
\ll_{C,\epsilon}
x^{\alpha+\epsilon}
,\]
with $\alpha=3/5$ (unconditionally) and $\alpha=1/2$
under the Lindel\"of Hypothesis (for details see
Kosyak et al.\,\cite{KMSZ}).}
\end{rem}

\section{A special case of the Bateman--Horn conjecture on average}
\label{s:momentstwo}
The goal of this section is to prove Theorem \ref{thm:main3}. 
Although the unconditional statement in Theorem \ref{thm:main3} is surpassed by the unconditional statement in Theorem \ref{thm:main2}, the proof of Theorem \ref{thm:main3} is, in a way, `orthogonal' to the one of Theorem \ref{thm:main2}; it thus has the potential of working in variations of the problem where the method 
behind Theorem \ref{thm:main2} would fail. Interestingly, like our prime gap criterion, it rests on a variation 
(implicit in Lemma \ref{withWilms}) of a 
certain very well studied problem involving prime numbers.
Both prime number questions are, however, quite different.
Lemma \ref{withWilms} allows one to show that many odd heights occur among
the ternary cyclotomic polynomials in a way different from
Theorem \ref{t.Eugenia}. 
\begin{Lem} 
\label{withWilms}
Let $h>1$ be odd. If there exists a prime
$p\ge 2h-1,$ such that the integer $q:=1+(h-1)p$
is a prime too, then $A(pqr)=h$ for some prime $r>q$.
 For $r$ one can take any prime
$r_1>q$ satisfying
$r_1(p+q)/2\equiv 1\,({\rm mod~}pq)$.
\end{Lem}
\begin{proof}
Define 
\begin{equation}
\label{eq:mpqr}
M(p;q)=\max_{r>q}\{A(pqr):2<p<q<r\}.
\end{equation}
Gallot et al.\,\cite[Theorem 43]{GMW} showed that 
if $q\equiv 1\,({\rm mod~}p),$ then
$$M(p;q)=\min\Big\{\frac{q-1}{p}+1,\frac{p+1}{2}\Big\}.$$
The conditions on $p$ and $h$ ensure that 
$M(p;q)=h.$ By \cite[Lemma 24]{GMW} it follows that $A(pqr_1)\ge h.$ This in combination with $M(p;q)=h$ shows that $A(pqr_1)=h$.
\end{proof}
In case $p\ge 2h+1$ the 
ternary cyclotomic polynomials from 
Lemma \ref{withWilms} are not optimal. We demonstrate
this in the case $h=63$ (with  $p=131$ and $q=8123$).
\begin{ex}
{\rm
Using the latter result and 
\cite[Lemma 24]{GMW}, we find that
$$A(131\cdot 8123\cdot 25497973)=\frac{8123-1}{131}+1=63$$
and $a_{131\cdot 8123\cdot 25497973}(13459462019674)=-63.$
}
\end{ex}
\medskip
We define the set 
$G\subset \mathbb N$ as follows, 
\begin{center}
$G:=\{m\in \mathbb N : 
\exists 
\ p  \,\in (4m,  32 m)$ 
such that $ 1+2mp $ is prime$\}$.
\end{center}
We would like to point out that the requirement $p< 32m$ is not 
necessary for the proof of Theorem \ref{thm:main3}, but is
needed in the proof of Theorem \ref{thm:main3+-}.

In the remaining part of this section
we show that  
the  
density of $G$ among all integers 
is positive,
i.e. that there exists $c_0>0$ such that
\begin{equation}\label{goal}
\liminf_{M\to+\infty}
\frac{\#\{m \in  G
\cap [1,M] \} }{M} \geq c_0.
\end{equation} 
For any natural number $m$
and any real number $x$ we define
\begin{center}
$ 
\pi_m(x) 
:=\#\left\{p\in [\frac{x}{2}\,,\,x) : 1+2mp \mbox{ is prime} \right\}$.
\end{center}
Further, for any $x\geq 0$ we 
define  
$$G(x) :=
\{m\in \mathbb N\cap [1,\,x/4):
\exists p \in (4m,x] \text{~such~that~}1+2mp \text{~is~prime}\}.$$
\begin{Lem}
\label{lem:dio}
For all $x,M\in \mathbb R$ 
with $x>8M $ and $M \geq 1 $
we have 
\begin{equation*}
\#\{m \in G(x)  
\cap (M/4, M]  
\} 
\sum_{1\le m\le M} \pi_m(x)^2 
\geq
\Big(
\sum_{M/4<m\le M} \pi_m(x)
\Big)
^2.
\end{equation*}
\end{Lem}
\begin{proof} 
Put \begin{equation*}
u_m(x)=
\left\{\begin{array}{ll}
1,&\text{if}\,\,\pi_m(x)\ne 0;\\
0,&\text{otherwise}.
\end{array}\right.
\end{equation*}
Fix $x>8M$.
By Cauchy's inequality
we have 
\begin{align*}
&\sum_{M/4< m\le M} \pi_m(x)
=
\sum_{M/4< m\le M} \pi_m(x)  u_m(x)
 \\
\leq&
\#\{M/4< m\leq M: \pi_m(x) >0 \}^{1/2}
\Big(
\sum_{1\le
m\le M} \pi_m(x)^2
\Big)
^{1/2}.
\end{align*}
If $m\le M$ and $p\geq x/2,$ then
$4m\le 4M <x/2 \leq 
p $, hence
$$\#M/4< m\leq M: \pi_m(x) >0 \}\leq \#\{m \in  G(x) 
 \cap  (M/4,M] 
 \},$$
concluding the proof.
\end{proof}

We would like to estimate the sums $\sum_{M/4< m\le M}\pi_m(x)$ and $\sum_{
1\le 
m\le M}\pi_m(x)^2$ 
appearing above. An upper bound, say $A$, 
for $\sum_{1\le m\le M}\pi_m(x)^2$ is easily obtained  by using standard sieve results.
Now if we could derive a lower bound $B$ for $\sum_{M/4< m\le M}\pi_m(x),$ then by Lemma \ref{lem:dio}
we get
$$\#\{m \in G(x)
\cap 
(M/4,M]  
\}\geq \frac{B^2}{A}.$$
Unfortunately the condition $x>8M$ makes it difficult to obtain a good lower bound for 
$\sum_{M/4< m\le M}\pi_m(x)$. We overcome this by using deep work of Bombieri, Friedlander and Iwaniec regarding the level of distribution 
of primes in  arithmetic progressions with fixed residue and varying moduli.

We start with estimating
$\sum_{1\le m\le M}\pi_m(x)^2$, for which we need the following 
lemma, which
is obtained on putting $b=k=l=1$ in \cite[Theorem 3.12]{HR}.

\begin{Lem}\label{thm:selbrgsiev}
Let $a$ be a positive even integer.
Then for all $x>1$ we have, uniformly in $a$, that
\begin{equation*}
\#\left\{p\leq x: ap+1 \mbox{ is prime} \right\} \leq  \frac{ 8 C_2 x }{(\log x )^2}
 \prod_{p \mid a 
 \atop
 p>2 }
\Big(
 \frac{p-1}{p-2} 
 \Big)
\Big\{
1+O
\Big(
\frac{\log \log x }{\log x }
\Big)\Big\},
\end{equation*}
where 
\begin{equation*}
C_2= \prod_{p>2} 
\Big(
1-\frac{1}{(p-1)^2} 
\Big)
\end{equation*}
is the twin prime constant.
\end{Lem}
\begin{rem}
\label{twin-primes}
{\rm
Hardy and Littlewood conjectured, based on heuristic reasoning, that asymptotically
\begin{equation*}
\#\left\{p\leq x: p+2 \mbox{ is prime} \right\} \sim 2C_{2}{\frac {x}{(\log x)^{2}}}.
\end{equation*}
A similar heuristic reasoning leads to the conjecture that asymptotically
$$\#\left\{p\leq x: ap+1 \mbox{ is prime} \right\}\sim
C_2
 \Bigg( \prod_{p \mid a 
 \atop
 p>2 }
 \Big(
 \frac{p-1}{p-2} 
 \Big)
  \Bigg)
\frac{    x }{(\log x )^2}
.$$
Both conjectures are special cases of the Bateman-Horn conjecture, cf.\ \cite{Al-ZomFukGarcia}.}
\end{rem}
\begin{Lem}\label{lem:classical sieve}
Let $x,M$ be any two positive real numbers.
Then
\[
\sum_{1\le m\le M} \pi_m(x)^2
 \leq
 64 C_1 C_2^2   M\frac{  x^2}{(\log x )^4}
\left\{1+O\left(\frac{\log \log x }{\log x } +\frac{1}{\sqrt{M}}\right)\right\}
,\] where  
\begin{equation*}
C_1:=
\prod_{p>2}
\left(1+
  \frac{2}{p(p-2)} + \frac{1}{p(p-2)^2}
  \right),
\end{equation*}
$C_2$ is the twin prime constant, and the implied
constant is absolute.
\end{Lem}
\begin{proof}
By Lemma \ref{thm:selbrgsiev}
with $a=2m$,
we get
\[
\pi_m(x)^2
\leq
8^2 C_2^2 \frac{   x^2 }{(\log x )^4}
 \prod_{ p \mid 2m \atop p>2 } \left( \frac{p-1}{p-2} \right)^2
\left\{   1+O\left(\frac{\log \log x }{\log x }\right)\right\}
, \] 
therefore, we
conclude
that
$\sum_{1\le m\le M} \pi_m(x)^2$ is at most
\[
8^2 C_2^2 \frac{  x^2}{(\log x )^4}
\left\{   1+O\left(\frac{\log \log x }{\log x }\right)\right\}
\sum_{1\le m\le M}
 \prod_{p \mid m \atop p>2 } \left( \frac{p-1}{p-2} \right)^2
.\]
We define the
multiplicative function $f$
via
\[
f(p^e):=
\mathds{1}_{p>2}(p)
\mathds{1}_{e=1}(e)
\left( \frac{2}{p-2} + \frac{1}{(p-2)^2}\right),
\
\
(e\in \mathbb N,  \, p \mbox{ prime})
.\]  One can easily verify
that
\[ \prod_{p \mid k \atop p>2  } 
\Big(
\frac{p-1}{p-2} 
\Big)^2
=\sum_{d\mid k }  f(d)=\sum_{d\mid k\atop 2\nmid d}  f(d)
\] for all non-zero
integers $k$.
This shows that
\begin{align*}
\sum_{1\le m\le M}
 \prod_{p \mid 2m \atop p>2  } \left( \frac{p-1}{p-2} \right)^2
 &=
\sum_{1\leq d \leq M \atop 2\nmid d}
  f(d)
\sum_{1\le m\le M \atop
  d \mid 2m}1
  \\
&=
M
\sum_{1\leq d \leq M}  \frac{ f(d)}{d}
 +O 
 \Big(
\sum_{1\leq d \leq M  }
  f(d) 
  \Big),
\end{align*}
where we used several times that $f(d)=0$ if $d$ is even.
Noting 
that
$f(p)\leq C/p$ for some absolute constant $C>0$
yields the bound
\[f(d) \leq \mu(d)^2 \frac{ C^{\omega(d) } }{d }  \ll
\frac{1}{\sqrt{d} }
,\ \  (d\in \mathbb N) , \]
which can be used to obtain 
\begin{align*}
\sum_{1\leq d \leq M }  \frac{ f(d)}{d}
&=
\sum_{d=1  
}^\infty   \frac{ f(d)}{d}  +O\left( \sum_{d> M}  \frac{1}{d^{3/2} } \right)
 =
C_1 +O\left(\frac{1}{\sqrt{M}} \right)
\end{align*}
and
\[
\sum_{1\leq d \leq M   }    f(d)
\ll  \sum_{1\leq d \leq M   }     \frac{1}{\sqrt{d}}
\ll \sqrt{M}
.\]
Putting everything together it follows that
\[
\sum_{1\le m\le M}
 \prod_{p \mid 2m \atop p>2 } \left( \frac{p-1}{p-2} \right)^2
=C_1M+O(\sqrt{M})
,\]
which is sufficient for our purposes.
\end{proof}

We now proceed to evaluate  the sum 
$\sum_{1\le m\le M} \pi_m(x)$
appearing in Lemma \ref{lem:dio}.
Writing
$n=1+2m p $ we see that it equals
\begin{eqnarray}
\label{eq:koukou}
 &&\sum_{M/4<   m\le M}
\sum_{ x/2  \leq p <  x \atop 1+2mp  \text{ prime}  }
1\nonumber
=\sum_{x/2  \leq p <  x  }
 \sum_{M/4< m\le M \atop 1+2mp  \text{ prime} }  1\nonumber\\
&=&
\sum_{ x/2  \leq p <  x }
\#\left\{n
\text{ prime}:  1+Mp/2 < n\leq 1 + 2Mp, \, n \equiv 1 ({\rm mod~}p) \right\}
\nonumber\\&\geq &
\frac{1}{\log (1+2Mx ) }
\sum_{ x/2  \leq p <  x}
\sum_{\substack{n
\text{ prime}  \\ 
1+M p/2  <  n \leq 1 + 2Mp  \\n 
\equiv 1 ({\rm mod~}p)  }  } \log n
,\end{eqnarray}
where we used that $ \log n
\leq \log (1+ 2Mp ) \leq \log (1+2Mx).$ 
\begin{rem}
\label{rem:1/2bombfriediwa}
{\rm
One now recognizes the argument in the latter sum as a  counting function of primes in 
an arithmetic progression of varying
modulus $p$, as $p$ runs through $[x/2,x).$
We would now use the Bombieri--Vinogradov theorem, however,
the   size of the primes $n$ is of the order of magnitude
$$1+2Mp \approx 2M x,$$
since the moduli $p$ have typical size $x.$
Thus, owing to
the condition $x>8M$,
we are counting primes in a progression
whose modulus exceeds the square-root of the size of the primes.
Therefore, the  Bombieri--Vinogradov theorem
cannot be applied in our case.
To be more precise, it can
 only be applied when the moduli are bounded by $\sqrt{z}/(\log z)^A$, where $A>0$ and $z$ is the length
 of the interval  $(0,z]$
we are counting primes in.
This means that we need
\[p \leq  \frac{\sqrt{ 2Mp }   }{(\log (2Mp))^A},\]
for some fixed $A>0,$ and this can only happen when $x=o(M)$. To deal with this problem 
we shall need a special case 
(Lemma \ref{thm:bomb_frid_iwa} below),
 of the work of
Bombieri, Friedlander and Iwaniec
\cite{MR891581}.
}
\end{rem}

As usual let 
$$  \theta (x;q,a):= 
\sum_{\substack{ p\leq x \\ p \equiv a ({\rm mod~}q) }}
 \log p,\,\,\,\,\psi (x;q,a):= \sum_{n\leq x \atop n \equiv a ({\rm mod~}q) } \Lambda(n),$$
with $\Lambda$ the von Mangoldt function.
\begin{Lem}[Bombieri--Friedlander--Iwaniec \cite{MR891581}]\label{thm:bomb_frid_iwa}For any $t \geq y \geq 3 $ we have $$\sum_{\sqrt{ ty }/2 \leq  q   <   \sqrt{ ty } }\left| \psi(t;q,1)-\frac{t }{\phi(q)}\right|\ll
t\left(\frac{\log y }{\log t }\right)^2 (\log \log t)^B ,$$ where $B$ is an absolute constant and the implied constant is absolute. 
\end{Lem} 
This estimate is obtained on 
setting $a\!=\!1, x\!=\!t$
and $Q\!=\!\sqrt{xy} $ in~\cite[Main Theorem, p.\,363]{MR891581}.
 \begin{Lem}\label{lem:consequence}For any $t \geq y \geq 3 $ with $y \leq t^{1/20}$
 we have $$\sum_{\substack{ q \text{ prime} \\
 \sqrt{ ty }/2 \leq  q   <   \sqrt{ ty } }}
 \left| \theta(t;q,1)-\frac{t }{\phi(q)}\right|
 \ll
t\left(\frac{\log y }{\log t }\right)^2 (\log \log t)^B ,$$ where $B$ is an absolute constant and the implied constant is absolute. \end{Lem}
\begin{proof}
Clearly
\[
\psi(t;q,1)=\theta(t;q,1)+ \sum_{k=2}^{\infty} 
\sum_{\substack{  p \leq t^{1/k} \\ p^k  \equiv 1 ({\rm mod~}q) }}
 \log p
 .\] The inner sum vanishes if $t^{1/k} < 2$, therefore only the integers
 $k \leq (\log t)/\log 2$ need to be taken into account. The contribution of all such integers 
 with $k\geq 3 $ is
$ \ll t^{1/3} \log  t $, since the sum over $p$    is     $\ll  t^{1/k}  $ by the prime number theorem.
 The steps so far are the standard arguments that one performs when moving from asymptotics for $\psi$
 to asymptotics for $\theta$, however, in our case, owing to the level of distribution 
 being comparable to the square root of the length of the interval, the term $k=2$ cannot be controlled with 
 the classical arguments.
 Instead, we use the bound  
 \[
 \frac{1}{\log t}\sum_{\substack{  p \leq \sqrt{t} \\ p^2  \equiv 1 ({\rm mod~}q) }}
 \log p
 \leq 
\sum_{\substack{ m \leq \sqrt{t} \\ m ^2  \equiv 1 ({\rm mod~}q) }}1
=
\sum_{\substack{ m \leq \sqrt{t} \\ m   \equiv -1 ({\rm mod~}q) }}1 
+
\sum_{\substack{ m \leq \sqrt{t} \\ m    \equiv 1 ({\rm mod~}q) }} 1
 , \]
  where we used the fact that $q$ is prime. Each of the sums in the right side 
 is trivially 
 $\ll \sqrt{t} /q + 1$ and therefore \[\sum_{\substack{  p \leq \sqrt{t} \\ p^2  \equiv 1 ({\rm mod~}q) }}
 \log p 
 \ll
 (\log t ) \left( \frac{\sqrt{t} }{q}+1\right).\]
 We thus find that 
 \[\psi(t;q,1)=\theta(t;q,1)+O\left(
  t^{1/3}(  \log t) 
  +
 \frac{\sqrt{t} }{q}   
 \log t  
  \right)
. \]
 This
 shows that the sum over $q$ in the statement of this lemma is  
\[ 
 \ll
 \sum_{ 
 \sqrt{ ty }/2 \leq  q   <   \sqrt{ ty }  }
 \left| \psi(t;q,1)-\frac{t }{\phi(q)}\right|
 +\sum_{ 
 \sqrt{ ty }/2 \leq  q   <   \sqrt{ ty } }
\left(
  t^{1/3}(  \log t)
  +
 \frac{\sqrt{t} }{q}   
 \log t  
  \right)
.\]
The first sum can be bounded by  
 Lemma~\ref{thm:bomb_frid_iwa}.
Noting that
$\sum_{x/2<q\le x}1/q=O(1),$
cf.\ \eqref{reciprocal}, we see that
the second sum is
 \[
 \ll
 \sqrt{ty} \,t^{1/3} (\log t ) + \sqrt{t} \log t, 
 \]
 which is $\ll t^{19/20} \ll t (\log t )^{-2},$ as $y \leq t^{1/20}$.
\end{proof}

\begin{Lem} \label{thm:corll}
Let $\psi: (1,\infty) \to (4,\infty) $
be any function satisfying
$\psi(M) \leq \log M$.
For any $M>1,$ we let $x= M \psi(M)$ and have
\[\sum_{M/4< m\le M} \pi_m(x)
\geq
\frac{M x}{
 2
\log (M x) }
\frac{\log 2 } {\log x }
\left\{1+
O\left(
\frac{(\log \log x)^{B+2}}{\log x}\right)
\right\}
,\] where $B$ is the absolute constant from  Lemma~\ref{lem:consequence}.
 \end{Lem}
\begin{proof}
The condition  
$p\in [x/2,x)$ in the
definition of $\pi_m(x)$ ensures that the interval 
$ (1+Mx/2, 1+M x ] $ is contained in the interval $ (1+Mp/2, 1+2M p ] $. Therefore, by  \eqref{eq:koukou} 
we see that the sum in our lemma is at least 
\[
\frac{1}{\log (1+2Mx) }
\sum_{ x/2 \leq  p  <  x }
\sum_{\substack{n  \text{ prime}  \\ 
1+Mx/2  <  n\leq  1+Mx
 \\ n\equiv 1 ({\rm mod~}p)  }  } \log n  . 
\]
Using Lemma \ref{lem:consequence}
with
$t=Mx$ and $y=\psi(M) $
shows that this is
\begin{equation*}
\frac{
  (1+Mx)-(1+M\frac{x}{2})   
 }{\log (1+2Mx) }
\sum_{ x/2 \leq  p  <   x}
\frac{1}{p-1}
+O\left(
\frac{Mx}{ \log( Mx) }
\left(
\frac{\log \psi(M) }{\log x}
\right)^2
(\log \log x )^B
\right).
\end{equation*}
Using the standard estimate
$$\sum_{p  \le  x } \frac{1}{p-1}= \log \log x +C'+O\left(\frac{1}{(\log x )^2}\right),$$
we obtain
\begin{equation}
\label{reciprocal}
\sum_{x/2<p\leq x } \frac{1}{p-1}
= 
\frac{\log 2 } {\log x }
\Big
\{1+
O\left(\frac{1}{\log x }\right)\Big\}.
\end{equation}
It follows that the main term is as claimed
in our lemma.
Furthermore, on using the bound $ \log\psi(M) \ll \log\log M \ll \log\log x,$
we see that
the error term is
\[\ll \frac{Mx}{ \log (Mx )  } 
\frac{  (\log \log x )^{B+2}  }{( \log x )^2}
,\]
as required.
\end{proof}

\begin{proof}[Proof of Theorem \ref{thm:main3}]
The first assertion is a corollary of Lemma \ref{withWilms}.

The inequalities obtained in
Lemmas \ref{lem:classical sieve}
 and \ref{thm:corll} with $\psi(M)=9$
in combination with the inequality in 
Lemma \ref{lem:dio} give rise, on
choosing $x= 1+8  M,$   
to  the inequality 
\[\#\{m \in G(1+8M) 
 \cap (M/4,M]  
 \}\,
64 C_1 C_2^2 M \frac{x^2 }{(\log x )^4}
 \geq\!
\left(
\frac{Mx}{\log (M x) } 
\frac{\log 2 } {2
\log x }
\right)^2\!\!\!
(1+o(1)).
\] 
In particular, the estimate $\log(Mx ) \leq 2 \log x $ yields
\[
 \#\{m \in G(1+8M) 
 \cap (M/4,M]  
 \}\, 
 \geq c' M 
(1+o(1) )
,\] where  \begin{equation*}
c'=\frac{(\log 2 )^2}{ 
1024
C_1   C_2^2  }>0.
\end{equation*}
Suppose that $m \in G(1+8M) 
 \cap (M/4,M]$. Note that since $M/4< m   $,  we have 
  $$p\leq x =1+8M <1+ 32 m,
  $$ and hence $p< 32 m $, therefore,
  the set $G(1+8M)  $ is contained in $G$.  
  We conclude that \eqref{goal}
holds with $c_0=c'$.
It follows that a positive proportion 
of all integers $m$ have the property that
there exists a prime $p>4m $ with also $1+2mp $ being a
prime. 
By Lemma \ref{withWilms}  we have $1+2m\in {\cal A}_t$ for 
each of those $m$, and it thus follows that unconditionally ${\cal A}_t$ contains a positive fraction of 
all odd natural numbers.
\end{proof}
\begin{rem} 
\label{rem:20m}
{\rm
The proof actually yields that 
a positive proportion 
of all integers $m$ have the property that
there exists a prime $4m<p < 32 m$ with also $1+2mp $ being a prime. This is what we will use in
the proof of Theorem \ref{thm:bound}}.
\end{rem}

\section{Some related issues}
\label{sec:referee}
\subsection{Estimating the smallest 
$n$ for which $A(n)=h$}
\begin{Def}
Given a natural number 
$h$, let $n_h$ be the smallest ternary integer, 
if it exists, such that $A(n_h)=h$. 
\end{Def}
The entries in the column $k/\varphi(pqr)$ in Table 1 suggest
the following question.
\begin{Quest} Let $h>1$ be
an integer. Does there exist an absolute constant
$0<c\le 1/2$ such that if $|a_n(k)|=h$, then $k>c\,\varphi(n_h)$?
\end{Quest}
A further 
question is to relate the size of
$n_h$ to $h$. 
See the final column of Table 1 for some numerical data. The 19th century estimate $A(pqr)\le p-1$ implies that $n_h\gg h^3$.
\begin{Con} 
\label{conjnh}
There are constants $E_1$ and $E_2$ such that
$h^{E_1}\ll n_h\ll h^{E_2}$ and $E_1\ge 3$.
\end{Con}
Theorem \ref{thm:bound} shows that for a positive fraction
of integers $h$ the upper bound 
in the conjecture holds true.
Its formulation involves Linnik's constant $L$.

\begin{Def}
Let $r\ge 0$ be an arbitrary fixed real number. For coprime integers 
$a$ and $d$, let $p_r(a,d)$ denote
the smallest prime $>d^r$ in the progression $a\,({\rm mod~}d)$.
\end{Def}
Linnik  proved  in 1944 that there exist positive constants $C$ and $L$ such that
$p_0(a,d)\le C\,d^L$.
The constant $L$ is known as 
\emph{Linnik's constant}. 
Xylouris \cite{Xylouris} proved that 
$L\le 5$, heavily relying on a fundamental paper by Heath-Brown \cite{HBLinnik}, who obtained $L\le 5.5$. On GRH Lamzouri 
et al.\ \cite{Lam} showed that $p(a,d)\le (\varphi(d)\log d)^2$ 
for $d>3$.

\par The following result generalizes Linnik's theorem.
\begin{Lem}
\label{generalLinnik}
Let $r>0$ be a real number. For coprime integers 
$a$ and $d$, let $p_r(a,d)$ denote
the smallest prime $>d^r$ in the progression $a\,({\rm mod~}d)$. Then there exists some absolute constant $C$ such that 
$p_r(a,d)\ll d^{r+C}$, where the implied constant 
is also absolute.
\end{Lem}

\begin{proof} We use Corollary 18.8 of the book of Iwaniec and Kowalski \cite{IK}. It states that there exists an explicit effectively computable constant
$L_1>0$ such that for all sufficiently large $d$ and all $x\geq d^{L_1}$ we have $$\psi(x;d,a) \gg \frac{x}{\varphi(d) \sqrt{d} } ,$$ where the implied constant is absolute.
Since $\varphi(d) \leq d $, this implies that  $$\psi(x;d,a) \gg \frac{x}{d^{3/2}  } .$$ For all $x>d^6$ we have 
$\sqrt{x} \leq x^{3/4} d^{-3/2}$ and hence, 
$$
\psi(x;d,a)-\theta(x;d,a)
\leq  
\psi(x )-\theta(x )\ll \sqrt{x}  \leq \frac{ x^{3/4} }{d^{3/2}}
.$$ 
Therefore,  if $x > d^{L_1+6}$ we deduce that 
 $$\theta(x;d,a) \gg \frac{x}{d^{3/2}  } ,$$ where the implied constant is absolute. To conclude our proof we note that
$p_r(a,d)$ is bounded by any real number  
 $x> d^r$   which satisfies 
 $$ \theta(x;d,a)
> \theta(d^r;d,a).$$
  Since $\theta(d^r;d,a) \leq \theta(d^r)
 <2 d^r$ by the prime number theorem, it suffices to find the least $x>d^r$ for which $  \theta(x;d,a)
\geq  2 d^r $. Clearly, this holds as long as $x>d^{L_1+6} $ and 
$x\,d^{-3/2} > C d^r $ for some large constant $C$. 
For both of these properties to hold it is sufficient that  $x\gg d^{r+6+L_1}$, from which we infer that 
$$
p_r(a,d) \ll d^{r+6+L_1}
,$$
with an absolute implied constant.
\end{proof}

The next result makes some progress
towards Conjecture \ref{conjnh}. It requires
only Linnik's theorem for its proof. Under
GRH the estimate holds with $L=2$.
\begin{Thm}
\label{thm:bound}
Let $\epsilon>0$.
Let $n_h$ be the smallest ternary integer, 
if it exists, such that $A(n_h)=h$.
There exists a constant $c_{\epsilon}>0$ such that  $n_h<c_{\epsilon}\,h^{3(L+1+\epsilon)}$ for 
a positive proportion of the odd natural numbers $h$.
\end{Thm}
\begin{proof}
Let $m$ be an integer 
such that there exists a prime $4m<p<32m$ 
with
also $q:=1+2mp$ being a prime. For any 
such $m$ we will show that $h:=1+2m\in \mathcal A_t$ and construct a ternary $n$ such that $A(n)=h$ and 
$n$ satisfies the required upper bound. Since, as we have
seen in the proof of Theorem \ref{thm:main3} (cf.\ Remark \ref{rem:20m}), there is a positive proportion of such $m$, the result follows.

We let $0<r_1<pq$ be the unique solution of
$r_1(p+q)/2\equiv 1\,({\rm mod~}pq)$.
If $r_1$ is even, we put
$r=p_0(r_1,pq)$. Note that $r>pq$. If $r_1$ is
odd, we let $s$ be the smallest prime
not dividing $r_1+pq$. 
Let $\delta>0$ be arbitrary.
Since the product of the primes not exceeding $x$ is of 
size $e^{(1+o(1))x}$, we conclude that $s<(pq)^{\delta}$
for all $m$ large enough. Observe that 
$r_1+pq$ and $spq$ are coprime. We
put $r=p_0(r_1+pq,spq)$. Note that $r>q$.
By Linnik's theorem we have
$r\le C(pq)^{(1+\delta)L}$.
By Lemma \ref{withWilms} we have $A(pqr)=h$.
Since $pqr>n_h$ and 
$$pqr=O(m\cdot m^2\cdot (m^{3(1+\delta)})^{L})=
O(h\cdot h^2\cdot (h^{3(1+\delta)})^{L})= O(h^{3(L+1+\epsilon)}),$$
with $\epsilon=\delta L$, the proof is completed.
\end{proof}
The next result can be seen as
a supplement to Theorem \ref{thm:main2}. The proof requires
Lemma \ref{generalLinnik} and a more precise version of Theorem
\ref{t.Eugenia} that is too long to be formulated here.
\begin{Thm}
\label{thm:bound2}
Let $t_{h}$ be the smallest 
optimal
ternary integer, if it exists, such
that $A(t_{h})=h$.
There exist positive constants $c$ and
$T$ such that  $t_{h}<c\,h^{T}$ for all $h\le x$ with at most
$\ll_{\epsilon} x^{3/5+\epsilon}$
exceptions.
\end{Thm}
\begin{proof}
We will use
\cite[Theorem 3.1]{Eugenia}, the full version of
Theorem \ref{t.Eugenia}. 
As Theorem \ref{t.Eugenia} is used in the proof of Theorem \ref{thm:main2}, we get the same
number of possible exceptions $h\le x$.
In terms of the $m$ of
Theorem \ref{t.Eugenia}, we have $l=2m-1$, with $l\le \sqrt{p}$.
We take $h=(p+l+2)/2$. The prime
$q$ indicated in the theorem is
bounded above by $p_2(a,p)$, with
$a$ an appropriate residue class.
The prime $r$ has to exceed $pq$
and be in an appropriate residue class modulo $pq$. By Lemma 
\ref{generalLinnik} we have
$pq\le pp_2(a,p)\ll p^{T_1}$ for 
some constant $T_1$. Thus by
Lemma \ref{generalLinnik} again, $r$ is $\ll p^{T_2}$ for some 
constant $T_2$.
Thus $pqr\ll p^{T_1+T_2}$. The result then follows with 
$T=T_1+T_2$ on noticing that $p=O(h)$.
\end{proof}

\subsection{Prescribed maximum or minimum coefficient}
So far we focused on possible heights of cyclotomic polynomials. Instead one can ask for possible maxima
and minima. In this section we will argue why the following conjecture is reasonable.
\begin{Con}
Each non-zero integer occurs either as the maximum or 
as the minimum coefficient of some cyclotomic polynomial.
\end{Con}
\begin{Def} We denote the maximum and minimum 
coefficients of
$\Phi_n$ by $A^+(n)$, respectively $A^-(n)$. 
We put ${\cal A}^+_{t}=\{
A^+(n)
:n\text{~is~ternary}\}$ and define
${\cal A}^-_{t}$ analogously.
We denote by $\mathcal A^{+}_{opt}$
the set of all $A^+(n),$ with $n$ optimal and define 
$\mathcal A^{-}_{opt}$ analogously.
\end{Def}
\begin{rem} {\rm Using the elementary identity 
$\Phi_n(1)=e^{\Lambda(n)}$ (valid for
$n>1$), we infer that}
$$A^-(n)=\begin{cases}
1 & \text{if~}n=p^k \text{~for some prime }p
\text{~and~} k\ge 1;\\
<0 & \text{otherwise}.
\end{cases}
$$
\end{rem}

Since our arguments rest on properties of ternary cyclotomic polynomials, the next
result due to Kaplan makes it plausible that asking 
which maximal coefficients can occur is in essence
the same as asking which possible minimum coefficients can occur.
\begin{Prop}{\rm (Implicit in Kaplan \cite{Kaplan}, explicit in Bachman and Moree \cite{BM})}.
\label{prop:Kaplan}
If $r,s>pq$, then
$$
A\{pqr\}=
\begin{cases}
A\{pqs\} & \text{~if~}s\equiv r\,({\rm mod~}pq);\\
-A\{pqs\} & \text{~if~}s\equiv -r\,({\rm mod~}pq).
\end{cases}
$$
\end{Prop}
This proposition can be used to prove the following lemma (recall
that $M(p;q)$ is defined in \eqref{eq:mpqr}).
\begin{Lem}
\label{lem:maxmin}
If $A(pqr)=M(p;q)$, then there exist
primes $r_1$ and $r_2$ such that 
$A^+(pqr_1)=M(p;q)$ and $A^-(pqr_2)=-M(p;q)$.
\end{Lem}
\begin{proof}
 The integers in $[-M(p;q),M(p;q)]\cap \mathbb Z$ are
precisely those that appear in $\Phi_{pqr}$ as $r$ ranges over the 
primes exceeding $q$; see
Gallot, Moree and
Wilms \cite[Proposition 1]{GMW}.
\end{proof}
In the proof of Theorem \ref{thm:main3} exclusively heights are considered
that are of the form $M(p;q)$. This observation together with  Lemma \ref{lem:maxmin}
then leads to a proof of the following variant of Theorem \ref{thm:main3}.
\begin{Thm}
\label{thm:main3+-}
If Conjecture \ref{Pi.2tuplet.1} holds true, then
${\cal A}^-_{t}\cup {\cal A}^+_{t}$ contains all 
odd integers. Unconditionally both
${\cal A}^-_{t}$ and ${\cal A}^+_{t}$
contain a positive fraction 
of all odd integers.
\end{Thm}
In our proof of Theorem \ref{thm:main1} we actually show that
$\mathcal R\subseteq \mathcal A_{opt}^{+}$ (recall that $\mathcal R$ is defined in \eqref{R}).
The optimal ternary
cyclotomic polynomials $\Phi_{pqr}$ used come from Theorem \ref{t.Eugenia}
and satisfy $r>pq$. This allows one then to invoke Proposition \ref{prop:Kaplan}
and conclude that $-\mathcal R\subseteq \mathcal A_{opt}^{-}$.

The following result is analogous to Theorem \ref{thm:main2}.
The proof of that result (given in \S\,\ref{sec:gaps})
rests on bounding above the integers $\le x$ that
are not in $\mathcal R$. Likewise the proof of Theorem 
\ref{thm:main2+-} rests on bounding above the integers
in $[-x,x]$ that are not in $-\mathcal R\cup \mathcal R$.
\begin{Thm}
\label
{thm:main2+-}
The set ${\cal A}^-_{opt}\cup {\cal A}^+_{opt}$ contains almost all 
integers.
Specifically, for any fixed $\epsilon>0$,
the number of integers with
absolute value $\le x$ that do not occur in 
${\cal A}^-_{opt}\cup {\cal A}^+_{opt}$
is $\ll_{\epsilon} x^{3/5+\epsilon}.$ Under the Lindel\"of Hypothesis this number
is $\ll_{\epsilon} x^{1/2+\epsilon}.$
\end{Thm}
\noindent Finally, we will derive a variant
of Theorem \ref{thm:main1}, namely Lemma 
\ref{l:Aopt+-}.

We put
\begin{eqnarray*}
{\mathcal R}^{\pm }&=&\Big
\{\frac{p-1}{2}- m: p{\rm ~is~a~prime}, \,m\ge 0,\, 4m^2+2m+3\le p\Big\}\nonumber\\
 &&\cup \,\,\Big
\{\frac{p-1}{2}+m: p{\rm ~is~a~prime}, \,m\ge 0,\, 4m^2+2m+3\le p\Big\}.
\end{eqnarray*}
We saw that $\mathcal R\subseteq \mathcal A_{opt}^{+}$ and
$-\mathcal R\subseteq \mathcal A_{opt}^{-}$. However, more 
is true.
\begin{Lem}
\label{l:Aopt+-}
We have $\mathcal R^{\pm}\subseteq \mathcal A^+_{opt}$ 
and $-\mathcal R^{\pm}\subseteq \mathcal A^-_{opt}$.
\end{Lem}
\begin{proof}
For the elements of $\mathcal R^{\pm}$ with $m=0$ this follows from Theorem \ref{t.Bach04}, for 
those with $m\ge 1$ it
follows from Theorem \ref{t.Eugenia} in combination with Proposition \ref{prop:Kaplan}.
\end{proof}

Taking $p=3,11,127$ and $m=0$ we see that $\{1,5,63\}$ are in 
${\cal R^{\pm}}$. This in
combination with 
Conjecture \ref{c.1.5.63} 
and Lemma \ref{l:Aopt+-} leads to the following conjecture.
\begin{Con}
\label{optR}
We have ${\cal R^{\pm}}={\mathbb N}$, ${\cal A}^+_{opt}={\mathbb N}$ and ${\cal A}^-_{opt}=-{\mathbb N}$.
\end{Con}

\subsection{Connection with Andrica's conjecture}
The aim of this subsection is to prove Theorem \ref{thm:Andrica}.
Our proof is a consequence of the following lemma that is
analogous to Lemma \ref{l.pi1}.
\begin{Lem} 
\label{l.pi1+-} Let $n\ge 5$ and 
$I_n:=\textstyle[\frac{p_n+1}{2},
\frac{p_{n+1}-1}{2}]$.\\
{\rm a)} If $p_{n+1}-p_n< \sqrt{p_n}+\sqrt{p_{n+1}},$ then
 $I_n\cap \mathbb N\subseteq \mathcal R^{\pm}.$\\
{\rm b)} If $p_{n+1}-p_n<\sqrt{p_n}+\sqrt{p_{n+1}}$ holds for $11\le p_n<2h$ with $h$ an
integer, then we have
$\mathbb N_h\subseteq \mathcal R^{\pm}.$
\end{Lem}
\begin{proof}
We let the integer $m_n$ be as in the proof 
of Lemma \ref{l.pi1}
and recall that $m_n\ge (\sqrt{p_{n}}-3)/2$.
Part a) 
follows if we can show that the final number $(p_n+1)/2+m_n$ is at least
$(p_{n+1}-1)/2-m_{n+1}-1$. 
Since both numbers are integers it suffices to require that 
$$\frac{p_n+1}{2}+m_n>\frac{p_{n+1}-1}{2}-m_{n+1}-2.$$ 
This is equivalent with
$d_n/2< m_{n}+m_{n+1}+3$.
Now our assumption on $d_n$ implies that
$$d_n/2< (\sqrt{p_{n}}-3)/2+
(\sqrt{p_{n+1}}-3)/2+3\le m_{n}+m_{n+1}+3,$$
as wanted.\\
b) This is a consequence of part a) and the observation that
$1,2,3,4$ and $5$ are in $\mathcal R^{\pm}$.
\end{proof}

\begin{proof}[Proof of Theorem 
\ref{thm:Andrica}]
A consequence of Lemma \ref{l.pi1+-} part b)
and the observation that we also have $p'-p<\sqrt{p}+\sqrt{p'}$ for
$p\le 11$.
\end{proof}

\newpage
\section{Ternary cyclotomic polynomials of small height}
\centerline{{\bf Table 1:} {\tt Ternary examples with prescribed height}}
\medskip
\medskip
\begin{center}
\begin{tabular}{|c|c|c|c|c|c|c|c|c|}
\hline
height & $p$ & $q$ & $r$ & $k$ & {\rm sign}& diff.&$\frac{k}{\phi(pqr)}$&
$\frac{\log(pqr)}{\log h}$\\
\hline
{\bf 1} & 3 & 7 & 11 & 0 & +&2&0&\\
\hline
{\bf 2} & 3 & 5 & 7 & 7 & --&{\bf 3}&0.146&6.714\\
\hline
{\bf 3} & 5 & 7  & 11 & 119 & --&{\bf 5}&0.496&5.418\\
\hline
{\bf 4} & 11 & 13 & 17 & 677 & -- & 7&0.353&5.623\\
\hline
{\bf 5} & 11 & 13 & 19 & 1008 &--&9&0.467&4.913\\
\hline
{\bf 6} & 13 &23  & 29 & 2499 &--&10&0.338&5.060\\
\hline
{\bf 7} & 17 &19  & 53 &6013 &+&14&0.402&5.009\\
\hline
{\bf 8} & 17 & 31 & 37 & 5596 &--&14&0.324&4.750\\
\hline
{\bf 9} &17  & 47 & 53 & 14538  &--&{\bf 17}&0.379&4.848\\
\hline
{\bf 10} & 17 & 29 & 41 & 4801 &--&{\bf 17}&0.267&4.305\\
\hline
{\bf 11} & 23 & 37 & 61 &20375 &--&16&0.428&4.527\\
\hline
{\bf 12} & 23 & 37 & 41 & 14471 &+&21&0.456&4.209\\
\hline
{\bf 13} & 31 & 59 & 73 & 58333 &--&25&0.465&4.601\\
\hline
{\bf 14} & 37 & 53 & 61 & 52286 &+&27&0.465&4.430\\
\hline
{\bf 15} & 37 & 47 & 61 &45939 &--&29&0.462&4.273\\
\hline
{\bf 16} & 41 & 79 & 97 & 133844  &--&30&0.446&4.565\\
\hline
{\bf 17} & 41 & 43 & 53 & 38240 &+&33&0.437&4.039\\
\hline
{\bf 18} & 61 & 97 & 103 & 178013 &--&34&0.302&4.608\\
\hline
{\bf 19} & 43 & 83 & 89 & 101051 &--&33&0.333&4.302\\
\hline
{\bf 20} &47 & 83 & 131 & 235842 &+&37&0.481&4.387\\
\hline
{\bf 21} & 47 & 101 &109  & 217278 &--&41&0.437&4.321\\
\hline
{\bf 22} & 53 & 83 & 89 & 165453 &--&44&0.441&4.166\\
\hline
{\bf 23} & 43 & 71 & 109 & 108355  &+&{\bf 43}&0.341&4.055\\
\hline
{\bf 24} & 53 & 103 & 109 & 189160 &--&42&0.330&4.183\\
\hline
{\bf 25} & 61 & 79 & 97 & 224640 &--&47&$0.500$&4.055\\
\hline
{\bf 26} & 41 & 71 & 97 & 96529 &--&{\bf 41}&0.359&3.852\\
\hline
{\bf 27} & 61 & 109 & 113 & 332589  &--&54&0.458&4.105\\
\hline
{\bf 28} & 53 & 89 & 131 & 186685 &--&{\bf 53}&0.314&4.001\\
\hline
{\bf 29} & 83 & 109 & 139 & 552035 &--&58&0.452&4.170\\
\hline
{\bf 30} & 67 & 131 & 137 & 389139 &--&52&0.333&4.116\\
\hline
{\bf 31} & 83 & 107 & 113 & 444435  &+&61&0.456&4.024\\
\hline
{\bf 32} & 79 & 149 & 163 &881529  &+&63&0.471&4.174\\
\hline
{\bf 33} & 73 & 103 & 113 &389314 &+&61&0.473&3.904\\
\hline
{\bf 34} & 71 & 109 & 113 & 409320 &--&60&0.483&3.879\\
\hline
{\bf 35} & 83 & 103 & 139 &544198  &--&69&0.471&3.934\\
\hline
{\bf 36} & 127 & 149 & 151 & 1246462 &--&72&0.445&4.148\\
\hline
{\bf 37} & 71 & 101 & 239 & 671716&+&67&0.403&3.975\\
\hline
{\bf 38} & 127 & 137 & 409 &3355658  &--&75&0.479&4.337\\
\hline
{\bf 39} & 83 & 149 & 157 & 941094 &+ &76&0.497&3.952\\
\hline
{\bf 40} & 79 & 233 & 239 & 1624556 & +&{\bf 79}&0.377&4.146\\
\hline
\end{tabular}
\end{center}
\vfil\eject
Table 1 gives the minimum ternary 
integer $n=pqr$ with $p<q<r$ such that $A(n)=m$ 
for the numbers $m=1,\ldots,40.$
The integer $k$ has
the property that $a_{pqr}(k)=\pm m,$ with the sign coming from the sixth column. The seventh column records the difference 
between the largest and smallest coefficient and is in
bold if this is optimal, that is, if the difference equals $p$ 
(compare Definition \ref{def:optimal}).
The second-to-last column gives the relative position of $k$ in $\Phi_{pqr}$. The final 
column gives, for $h>1$, the exponent $e$ such that $pqr=h^e$.

The heights $h$ in Table 1 satisfy
$h\le 2p/3$ with equality
only in case $h=2$. This is consistent with the generalized Sister Beiter conjecture due to Gallot and Moree \cite{GM}. \\

\noindent {\bf Acknowledgement.} 
The authors thank Danilo Bazzanella, Adrian Dudek,
Tom\'as Oliveira e Silva, Alberto Perelli and Tim Trudgian 
for helpful
email correspondence. 
Olivier Ramar\'e kindly provided us with a
high accuracy evaluation of $C_1.$
We thank the referee for excellent remarks that
helped to improve the exposition of this paper and
questions that led to
the addition of Section \ref{sec:referee}. 

In case an integer $h$ is not in $\mathcal R$ (defined in \eqref{R}), 
still the work of Moree and Ro\c su \cite{Eugenia} offers
some hope to show that $h$ occurs as a height (as we saw
in case $h=63$). To make
this more precise involves understanding the distribution of
inverses modulo primes. We thank Cristian Cobeli for sharing
some observations and numerical experiments  on this.

The first author is a novice in number theory and is very
grateful to Pieter Moree for introducing him to the field. The authors, except the fourth, are or
were supported by
the Max Planck Institute for Mathematics and 
    thankful for this.
The fourth author is supported
by  the National Natural Science Foundation of China (Grant No. 11801303),
project ZR2019QA016 supported by the Shandong Provincial Natural Science 
Foundation and
a project  funded by the China Postdoctoral Science Foundation (Grant 
No. 2018M640617).

\medskip\noindent {\footnotesize  Institute of Mathematics, \\
Ukrainian National Academy of Sciences,\\
3 Tereshchenkivs'ka Str., 
01024 Kyiv, Ukraine.\\   
e-mail: {\tt kosyak02@gmail.com}}\\

\medskip\noindent {\footnotesize Max-Planck-Institut f\"ur Mathematik,\\
Vivatsgasse 7, D-53111 Bonn, Germany.\\
e-mail: {\tt moree@mpim-bonn.mpg.de}}\\

\medskip\noindent
{\footnotesize School of Mathematics and Statistics,\\ 
University of Glasgow,\\ 
University Place, Glasgow, G12 8SQ, United Kingdom.\\
E-mail: {\tt efthymios.sofos@glasgow.ac.uk}}\\

\medskip\noindent
{\footnotesize School of Mathematical Sciences,\\ 
Qufu Normal University,\\ 
Qufu 273165, P. R. China.\\
E-mail: {\tt zhangbin100902025@163.com}}
\vskip 5mm
\end{document}